\documentclass{article}
\usepackage{amsmath,amssymb,amsthm,graphicx,epsfig,float,url}
\usepackage[colorlinks=true,citecolor={black},linkcolor={black}]{hyperref}
\usepackage{pdfsync}
\usepackage[usenames, dvipsnames]{xcolor}
\usepackage{tikz}
\usepackage{subfig}
\usepackage{bbm}
\usepackage{stmaryrd}
\usepackage{amssymb}
\usepackage{mathrsfs}  
\usepackage{caption}
\usepackage{float}
\usepackage{esint}

\usepackage{dsfont}
\usepackage[makeroom]{cancel}

\newcommand{\R}{\textnormal{I\kern-0.21emR}}
\newcommand{\N}{\textnormal{I\kern-0.21emN}}

\renewcommand{\geq}{\geqslant}
\renewcommand{\leq}{\leqslant}

\def\e{{\varepsilon}}

\allowdisplaybreaks

\def\Yint#1{\mathchoice
    {\YYint\displaystyle\textstyle{#1}}%
    {\YYint\textstyle\scriptstyle{#1}}%
    {\YYint\scriptstyle\scriptscriptstyle{#1}}%
    {\YYint\scriptscriptstyle\scriptscriptstyle{#1}}%
      \!\iint}
\def\YYint#1#2#3{{\setbox0=\hbox{$#1{#2#3}{\iint}$}
    \vcenter{\hbox{$#2#3$}}\kern-.51\wd0}}
\def\longdash{{-}\mkern-3.5mu{-}} 
\def\tiltlongdash{\rotatebox[origin=c]{15}{$\longdash$}}
\def\fiint{\Yint\tiltlongdash}

\usepackage[dvipsnames]{xcolor}

\newtheorem*{theorem*}{Theorem}

\newtheorem{theorem}{Theorem}

\newtheorem{material}{material}

\newtheorem{definition}[material]{Definition}
\newtheorem{lemma}[material]{Lemma}

\newtheorem{remark}[material]{Remark}

\def\O{{\Omega}}

\def\n{{\nabla}}

\def\p{{\varphi}}

\def\TT{{(0;T)\times \T}}
\def\TTo{{(0;T)\times \mathbb{T}}}

\def\T{{\mathbb T^d}}

\usepackage{xargs}% use more than one optional parameter in a new command
 \usepackage[colorinlistoftodos,textsize=small]{todonotes}
 \newcommandx{\christian}[2][1=]{\todo[linecolor=red,backgroundcolor=red!25,bordercolor=red,#1]{#2}}
 \newcommandx{\laura}[2][1=]{\todo[linecolor=blue,backgroundcolor=blue!25,bordercolor=blue,#1]{#2}}
 \newcommandx{\info}[2][1=]{\todo[linecolor=green,backgroundcolor=green!25,bordercolor=green,#1]{#2}}
 \newcommandx{\improvement}[2][1=]{\todo[linecolor=yellow,backgroundcolor=yellow!25,bordercolor=yellow,#1]{#2}}
 
  \newcommandx{\biblio}[2][1=]{\todo[linecolor=blue,backgroundcolor=magenta!25,bordercolor=blue,#1]{#2}}

 \numberwithin{equation}{section}
\begin{document}

\title{Another look at qualitative properties of eigenvalues using effective Hamiltonians}

%    Remove any unused author tags.

%    author one information
\author{Idriss Mazari-Fouquer}
\date{ \today}

\maketitle
\begin{abstract}
The goal of this paper is to review several qualitative properties of well-known eigenvalue problems using a different perspective based on the theory of effective Hamiltonians, working exclusively on the Hopf-Cole transform of the equation. We revisit some monotonicity results as well as the derivation of several scaling limits by means of the Donsker-Varadhan formula, and we point out several differences between the case of quadratic Hamiltonians and non-quadratic ones.   
\end{abstract}

\paragraph{Acknowledgement:}This work was started during a visit to Adrian Lam at the Ohio State University (Columbus); the hospitality of the institution and its support are acknowledged.  The author wishes to thank P. Cardaliaguet, A. Lam and S. Liu for numerous fruitful conversations, \textcolor{black}{as well as the anonymous referee of this article for insightful comments that helped improve the paper}. The author was supported by a PSL Young Researcher Starting Grant 2023 from Paris Dauphine Universit\'e PSL. 
\tableofcontents
\section{Introduction}

The goal of this article is to survey several important qualitative results in the theory of principal eigenvalues for (possibly non-symmetric) linear operators that appear, for instance, in the study of population dynamics, and have recently attracted a lot of attention. In particular, we will show how several results follow from the Donsker-Varadhan formulation of eigenvalues; this gives a natural reinterpretation of some functional approaches to the qualitative properties of eigenvalue, see Section \ref{Se:Biblio}. 
\paragraph{Notational conventions}
\begin{itemize}
\item For a measurable set $\O$, $\mathcal P(\O)$ denotes the set of probability measures on $\O$. When $\eta\in \mathcal P(\O)$ has a $\mathscr C^k$ density with respect to the Lebesgue measure, we identify the measure and the density (and in particular we speak of $\mathscr C^k$ probability measures for a measure with $\mathscr C^k$ density).
\item For a measurable set $\O$ and an integrable function $f\in L^1(\O)$, the notation $\fint_\O f$ denotes the average value of $f$: $\fint_\O f=(1/|\O|)\int_\O f$.
\item $\T$ denotes the $d$ dimensional torus. For a given $T>0$, $\alpha\,, \beta \in \N$, $\mathscr C^{\alpha,\beta}_{\mathrm{per}}(\TT)$ denotes the set of $T$-periodic in $t$, $\T$-periodic in $x$ functions that are $\mathscr C^\alpha$ in $t$ and $\mathscr C^\beta$ in $x$. 
\item For $T>0$, $\alpha\in \N$, the set $\mathscr C^\alpha_{\mathrm{per}}((0;T))$ denotes the set of $T$-periodic $\mathscr C^\alpha$ functions in $\mathscr C^\alpha((0;T))$ and $L^\infty_{\mathrm{per}}((0;T))$ denotes the set of $T$-periodic $L^\infty$ functions.
\item The integration variables are implicit and given by the integrals, meaning that $\int_\O f$ refers to the integration in the $x$ variable, $\int_0^T f$ to the integration in the $t$ variable and $\iint_\TT f$ to the integration in the $(t,x)$ variable.
\end{itemize}
\subsection{Main setting}
\paragraph{Effective Hamiltonians in the torus}
Throughout this article, we work in $\T$ and with a fixed time horizon $T>0$. We let $H=H(p)$ be a $\mathscr C^2$ Hamiltonian that satisfies
\begin{equation}\label{Eq:HypH}\begin{cases}
\forall p\in \R^d\setminus\{0\}\,, H(p)> H(0)=0\,, \\\\\forall p\neq 0\,, \n^2_{pp}H(p)\text{ is symmetric positive definite}\,, \\
\\ \text{ There exists a continuous, coercive, convex, non-negative function $\mathcal H$ such that}\\\text{ for all $a>0$ $\mathcal H(a)>\mathcal H(0)=0$ and, }\\\text{for any $f\in \mathscr C^2(\T)$, $\mathcal H\left( \Vert f-\fint_\T f\Vert_{L^1(\T)}\right)\leq \int_\T H(\n_xf).$}
\\
\\ \exists \beta\in (0;1)\,, \forall a\in (0;1)\,, \forall p\in \R^d\,,  a H(p)\geq H(a^\beta p).\end{cases}\end{equation} Observe that the fourth assumption implies the superlinearity of $H$: \[ \lim_{\Vert p\Vert\to \infty}\frac{H(p)}{\Vert p\Vert}=+\infty.\] These properties are satisfied for $H:p\mapsto \vert p\vert^r$ for any $r>1$ (the third property is just the Poincar\'e inequality).

For a fixed $m\in \mathscr C^{1,2}_{\mathrm{per}}(\TT)$, let $(\lambda_H(m),\p_H)$ be the unique eigenpair associated with $\partial_t-\Delta+H(\n_x\cdot)+m$, that is, the unique solution of 
\begin{equation}\label{Eq:MainHamiltonian}\begin{cases}
\lambda_H(m)+\partial_t \p_H-\Delta \p_H+H(\n \p_H)=-m&\text{ in }\TT\,, 
\\ \p_H(T,\cdot)=\p_H(0,\cdot)\,, 
\\ \fiint_\TT \p_H=0,\end{cases}\end{equation}where $\p_H\in \mathscr C^{1,2}_{\mathrm{per}}(\TT)$.
That such a couple exists and is uniquely defined is well-known, as is its link with the long-time behaviour of Hamilton-Jacobi equation. We refer to \cite{zbMATH01578865}. The aim of this paper is to investigate the qualitative properties of $\lambda_H$ and the influence of the different parameters of the equation (typically, replacing $\partial_t-\Delta+H$ with $\tau\partial_t -\mu\Delta+\e H$). Starting with the works of Beltramo \& Hess \cite{zbMATH03897446}, several contributions covered a number of these properties and they often rely on delicate construction of sub and super solutions. Many of these works were motivated by applications of eigenvalue problems to population dynamics \cite[Chapter 2]{zbMATH07668634}. We will show that several of these approaches can be either streamlined or re-interpreted naturally within the framework of the Donsker-Varadhan formula. Let us also observe that we show in a last section how to recover some other well-known results for elliptic problems with advection (the parabolic setting corresponding to a degenerate elliptic operator with advection).

We believe the approach we present here is more flexible and amenable to generalisation, and that the interest also lies in providing another outlook on questions of interest. Before we outline it, let us present the link with the usual eigenvalue problem.

\paragraph{Link with linear eigenvalue problems}
A core problem in the study of semi-linear parabolic equations is the following eigenvalue equation: for a given $m\in \mathscr C^{1,2}_{\mathrm{per}}(\TT)$, consider the principal eigenpair $(\lambda(m),u_m)$ associated with $\partial_t-\Delta-m$ or, in other words, the unique solution of 
\begin{equation}\label{Eq:Main}\begin{cases}
\partial_t u_m-\Delta u_m=mu_m+\lambda(m)u_m&\text{ in }\TT\,, 
\\ u_m(T,\cdot)=u_m(0,\cdot)&\text{ in }\T\,, 
\\ u_m\geq 0\,, u_m\neq 0\,, \fiint_\TT u_m^2=1.\end{cases}\end{equation} The existence and uniqueness of this eigenpair is a consequence of the Krein-Rutman theorem. For the link between \eqref{Eq:Main} and reaction-diffusion equations we refer to \cite[Chapter 2]{zbMATH07668634}. From the strong maximum principle, we can use the Hopf-Cole transform $u_m=e^{-\p_m}$, where $\p_m$ solves 
\[ \lambda(m)+\partial_t\p_m-\Delta \p_m+\vert \n_x \p_m\vert^2=-m.\]In other words, $(\lambda(m),\p_m)$ can be interpreted as the solution of \eqref{Eq:MainHamiltonian} for the quadratic Hamiltonian $H:p\mapsto \vert p\vert^2$. It is important to keep this relation in mind. Of course, several properties (\emph{e.g.} monotonicity in frequency,  in diffusivity etc)  are well-known for $\lambda(m)$, see \cite{zbMATH07668634}, and we will investigate to which extent these results hold for general Hamiltonians. In doing so, we will be led to underlining the exceptionality of the quadratic Hamiltonian $H(p)=|p|^2$ in the class of Hamiltonians.

\textcolor{black}{
\begin{remark}
It should be noted that, for general Hamiltonians $H$, we are not aware of any possibility to tie the effective Hamiltonian to the principal eigenvalue of a linear elliptic operator, as can be seen by looking for a linear operator $L$ and a non-linearity $f$ such that, letting $\p=f(u)$, $u$ should be an eigenfunction of $L$.  Furthermore, in Theorem \ref{Th:Reversibility}, the quadratic Hamiltonian is shown to be the only one having a particular property, that of reversibility (see Definition \ref{De:Reversibility}). For the quadratic Hamiltonian, this property is proved by exploiting the link with \eqref{Eq:Main}.
\end{remark}
}

\paragraph{Donsker-Varadhan formulation of the effective Hamiltonian, invariant measure}
Variational formulations are crucial in spectral problems. Naturally, for non-symmetric operators, the best we can hope for is a saddle-point formulation. Here, this is expressed by the Donsker-Varadhan formula \cite{zbMATH03549753,zbMATH03699420}:
\begin{equation}\label{Eq:DonskerVaradhan}
-\lambda_H=\max_{\eta\in \mathcal P(\TT)}\min_{\varphi\in \mathscr C^{1,2}_{\mathrm{per}}(\TT)} \iint_\TT \left(\partial_t\p-\Delta\p+H(\n_x\p)+m\right)d\eta.\end{equation}It is straightforward (see Lemma \ref{Le:QuantifiedDV} below) to see that there exists a unique saddle point $(\eta_H,\p_H)$ of this problem with $\fiint_\TT \p_H=0$.
\begin{definition}\label{De:InvariantMeasure}
The couple $(\lambda_H,\p_H)$ is called the eigenpair attached to the operator $\partial_t-\Delta+H(\n_x\cdot)+m$. $\lambda_H$ is either called the effective Hamiltonian or the eigenvalue associated with $\partial_t-\Delta+H(\n_x\cdot)+m$. The measure $\eta_H\in \mathscr C^{1,2}_{\mathrm{per}}(\TT)$ is called the associated invariant measure and satisfies
\begin{equation}\label{Eq:DefInvariantMeasure}
\begin{cases}
-\partial_t\eta_H-\Delta \eta_H-\nabla\cdot(\eta_H\n_pH(\n_x\p_H))=0&\text{ in }\TT\,, 
\\ \eta_H(T,\cdot)=\eta_H(0,\cdot)&\text{ in }\T\,, 
\\ \eta_H\geq 0\,, \int_\T \eta_H(t,\cdot)=\frac1T\text{ for any $t\in [0;T]$}.
\end{cases}
\end{equation}\end{definition}

\textcolor{black}{
The wording ``invariant measure" and Eq. \eqref{Eq:DefInvariantMeasure} come from the optimal control interpretation of $\lambda_H(m)$, which is detailed in the next paragraph (see Remark \ref{Re:Interpretation}). }

\begin{remark}[Regarding terminology]
We use  the wordings ``effective Hamiltonian" and ``eigenvalue" indifferently, as they are both standard in different settings. ``Effective Hamiltonian" is more appropriate when dealing with the long-time behaviour of Hamilton-Jacobi equations, while ``eigenvalues" is more convenient when handling the quadratic case, in which case $\lambda_H$ is indeed the principale eigenvalue of a linear operator.
\end{remark}

\begin{remark}[On \eqref{Eq:DonskerVaradhan}]
There are various min-max formulations for eigenvalues, several of which are reviewed in \cite{zbMATH07183857}, where they are put to use to study the monotonicity of eigenvalues when increasing the magnitude of incompressible advection terms. In particular, we emphasise that the core estimate of \cite{zbMATH07183857} is a stability estimate for \eqref{Eq:DonskerVaradhan}, which can be obtained in a more direct manner using the invariant measure $\eta_H$. The generalisation of the approach of \cite{zbMATH07183857} was our initial motivation for the present work. 
\end{remark}

\paragraph{The optimal control interpretation of \eqref{Eq:MainHamiltonian}}
Another interest of the Hopf-Cole transform is that it allows to reformulate the eigenvalue $\lambda(m)$ in terms of an optimal control problem. To be more specific, we let, for any Hamiltonian $H$ satisfying \eqref{Eq:HypH}, $L$ denote the Legendre transform of $H$:
\[ L(\alpha)=\sup_{p\in \R^d}\left(\langle p,\alpha\rangle-H(p)\right).\]
Then there holds
\begin{multline}\label{Eq:OptimalControl} -\lambda_H(m)=\max_{\alpha \in L^\infty_{\mathrm{per}}(0,T;\R^d)}\iint_\TT \left(-L(\alpha)+m\right)d\eta_\alpha\\\text{ with $\eta_\alpha$ the unique solution of }\begin{cases}-\partial_t\eta_\alpha-\Delta \eta_\alpha-\nabla\cdot(\alpha \eta_\alpha)=0\,, 
\\ \eta_\alpha(T,\cdot)=\eta_\alpha(0,\cdot)\,, 
\\ \eta_\alpha\geq 0\,, \int_\T \eta_\alpha(T,\cdot)=\frac1T.\end{cases}\end{multline}
\begin{proof}[Proof of \eqref{Eq:OptimalControl}]
Let for any $\alpha\in L^\infty_{\mathrm{per}}(0,T;\R^d)$ $J(\alpha):=\iint_\TT (m-L(\alpha))d\eta_\alpha$ and let $(\lambda_H,\p_H)$ be the solution of \eqref{Eq:MainHamiltonian}. Multiplying \eqref{Eq:MainHamiltonian} by $\eta_\alpha$, integrating by parts and using the equation satisfied by $\eta_\alpha$ yields
\begin{align*}
-\iint_\TT md\eta_\alpha&=\lambda_H(m)+\iint_\TT H(\n_x\p_H)d\eta_\alpha+\iint_\TT \p_H(-\partial_t\eta_\alpha-\Delta\eta_\alpha)
\\&=\lambda_H(m)+\iint_\TT H(\n_x\p_H)d\eta_\alpha-\iint_\TT \langle \alpha,\n_x\p_H\rangle d\eta_\alpha
\end{align*}
whence 
\[J(\alpha)=-\lambda_H(m)+\iint_\TT (-L(\alpha)-H(\n_x\p_H)+\langle\alpha,\n_x\p_H\rangle)d\eta_\alpha.\]
By definition of $L$, $-L(\alpha)-H(\n_x\p_H)+\langle\alpha,\n_x\p_H\rangle\leq 0$, with equality if $\alpha=\n_pH(\n_x\p_H)$. As $\eta_\alpha\geq 0$ the conclusion follows.
\end{proof}
\textcolor{black}{
\begin{remark}[Interpreting the invariant measure]\label{Re:Interpretation}
The invariant measure has a simple interpretation as the measure describing the large time behaviour of agents engaged in  an optimal control problem, the large-time asymptotics of which is given by \eqref{Eq:OptimalControl}: given a fixed distribution of resources $m$ that varies periodically in time and in space, and given an initial density of players, say $\bar\eta_0$, a central planner is directing each agent to move according to a certain strategy, say $\alpha$, so as to maximise the amount of resources collected by the entire population, while paying a certain cost $L(\alpha)$. It is possible to prove, using standard techniques in Hamilton-Jacobi equations (and adapting for instance \cite{zbMATH02130241}), that, when this optimal control problem takes place in a very large time interval, the density of players will converge to the density $\eta_H$ of the invariant measure. Thus, $\eta_H$ represents a fully adapted density of players.
\end{remark}}

\paragraph{Properties under investigation and interpretation}
We study several limiting regimes and properties for effective Hamiltonians in the following sense: considering, for $\tau,\mu,\e>0$ the effective Hamiltonian $\lambda_H(\tau,\mu,\e)$ associated with $\tau\partial_t-\mu\Delta+\e H(\n_x\cdot)+m$, we want to understand the behaviour of $\lambda_H(\tau,\mu,\e)$ as $\tau\,, \mu\,, \e$ vary. This includes:
\begin{itemize}
\item Understanding the limit $\tau\,,\mu\,, \e\to 0$ or $+\infty$. Several of these limits are well-known in the quadratic case, where they are obtained by working directly on \eqref{Eq:Main}, typically by means of sub and super-solutions. Our approach is different in nature, and we recover several well-known results.
\item Understanding the monotonicity properties of $\lambda$ with respect to the different parameters. This line of research is more recent, and the first major breakthroughs were obtained by Liu, Lou, Peng \& Zhou \cite{zbMATH07122717}. We revisit their approach, and provide a different perspective on their proof, in particular highlighting the importance of the reversibility of the quadratic Hamiltonian (see Definition \ref{De:Reversibility}).
\end{itemize}
Throughout and unless stated otherwise, $H$ satisfies \eqref{Eq:HypH} and $m\in \mathscr C^{1,2}_{\mathrm{per}}(\TT)$.

\paragraph{One remark regarding the domain we work on}
Although all of our analysis is done in the case of the torus $\T$, the same results and methods hold when working in a fixed, compact and smooth domain of $\R^d$ (in space), retaining the periodicity conditions in time but replacing the periodicity conditions in space with Neumann boundary conditions. 

\subsection{Basic qualitative properties of effective Hamiltonians}
We begin with basic estimates on effective Hamiltonians that were originally obtained by Hutson, Shen \& Vickers \cite{zbMATH01579419}  in the quadratic case .
\begin{theorem}\label{Th:Basic}
There holds
\begin{equation}\label{Eq:BoundL}
-\Vert m\Vert_{L^\infty(\TT)}\leq \lambda_H(m)\leq\xi\left(\fint_0^Tm(t,\cdot)dt\right)\end{equation} where $\xi\left(\fint_0^Tm(t,\cdot)dt\right)$ is the first (elliptic) eigenvalue associated with $-\Delta+H(\n_x\cdot)+\fint_0^T m(t,\cdot)dt$. The left-hand side inequality is an equality if, and only if, $m$ is constant. The right-hand side inequality is an equality if, and only if, $m$ does not depend on $t$.\end{theorem}
We then show a generalisation of a result of Beltramo \& Hess \cite[Proposition 15.4]{zbMATH00049232}. While the proof of this result (originally given in \cite{zbMATH03897446}) relies on a delicate comparison with a Dirichlet eigenvalue, we provide a different approach based on the optimal control interpretation \eqref{Eq:OptimalControl} of $\lambda_H$.
\begin{theorem}\label{Th:BeltramoHess}
Let $m$ be such that 
\begin{equation}\label{Eq:AssumptionBeltramoHess} \int_0^T \left(\max_{x \in \T}m(t,x)\right)dt>0.\end{equation} Let for any $\e>0$ $\lambda_H(\e)$  be the effective Hamiltonian associated with $\partial_t-\Delta+H(\n_x\cdot)+\frac1\e m$. Then 
\[ \lambda_H(\e)\underset{\e\to 0}\rightarrow -\infty.\]
\end{theorem}

\subsection{Some scaling limits of effective Hamiltonians}

\paragraph{The low and large frequency limits}
In this section we let for any $\tau>0$ $\lambda_H(\tau)$ be the effective Hamiltonian associated with the operator $\tau\partial_t-\Delta+H(\n_x\cdot)+m$.
%\[\lambda_H(\tau)+ \tau\partial_t \p_{H,\tau}-\Delta \p_{H,\tau}+H(\n_x \p_\tau)+m=0.\]
 We begin by providing an alternative proof of the results of Liu, Lou, Peng \& Zhou \cite{zbMATH07122717}.
\begin{theorem}\label{Th:AsymptoticsFrequency}There holds \begin{equation}\label{Eq:LargeFrequency}
\lim_{\tau\to \infty}\lambda_H(\tau)=\xi\left(\fint_0^Tm(t,\cdot)dt\right)\end{equation} where $\xi\left(\fint_0^Tm(t,\cdot)dt\right)$ is the  (elliptic) eigenvalue associated with $-\Delta+H(\n_x\cdot)+\fint_0^T m(t,\cdot)dt$ and \begin{equation}\label{Eq:LowFrequency}
\lim_{\tau \to 0}\lambda_H(\tau)=\fint_0^T \xi(m(t,\cdot))dt
\end{equation} where for any $t\in [0;T]$ $\xi(m(t,\cdot))$ is the  (elliptic) eigenvalue associated with $-\Delta+H(\n_x\cdot)+ m(t,\cdot)$.
\end{theorem}
\paragraph{The low and large  diffusion limits }
\begin{theorem}\label{Th:LowDiffusion2}
Let for any $\mu>0$ and any $m\in \mathscr C^{1,2}_{\mathrm{per}}(\TT)$ $\lambda_H(m,\mu)$ be the first eigenvalue of $\partial_t+\mu(-\Delta+H(\n_x\cdot))+m$. Then 
\[ -\lambda_H(m,\mu)\underset{\mu\to 0}\rightarrow \max_{x\in \T}\fint_0^T m(t,x)dt\] and 
\[ -\lambda_H(m,\mu)\underset{\mu\to \infty}\rightarrow \fiint_\TT m.\]
\end{theorem}
\paragraph{The large heat operator limit}In the case of a large heat operator, we can actually obtain a more precise asymptotic expansion of the eigenvalue.
\begin{theorem}\label{Th:LargeHeat}
Let for any $\e>0$ and any $m\in \mathscr C^{1,2}_{\mathrm{per}}(\TT)$ $\lambda_H(m,\e)$ be the first eigenvalue of $\partial_t-\Delta+\e H(\n_x\cdot)+m$. Then 
\begin{equation*} \frac{-\lambda_H(m,\e)-\fiint_\TT m}\e \underset{\e\to 0}\rightarrow \fiint_\TT H(\n_x \psi_0)\end{equation*} where \begin{equation}\label{Eq:Lim}\begin{cases}\partial_t \psi_0-\Delta \psi_0=\fiint_\TT m -m&\text{ in }\TT\,, 
\\ \psi_0(T,\cdot)=\psi(0,\cdot)&\text{ in }\T\,, 
\\ \fiint_\TT \psi_0=0\end{cases}.\end{equation}
\end{theorem}
This result will be used to distinguish between quadratic and non-quadratic Hamiltonians, see the next section.

\subsection{A key difference between quadratic and non-quadratic Hamiltonians}
In the upcoming sections we will investigate the monotonicity of effective Hamiltonians with respect to various parameters, and we argue that a key part in the proof of these properties is the reversibility of effective Hamiltonians in the following sense:
\begin{definition}\label{De:Reversibility} Let, for any $m\in\mathscr C^{1,2}_{\mathrm{per}}(\TT)$, $\lambda_{H,+}(m)$ be the  eigenvalue associated with $\partial_t-\Delta+H(\n_x\cdot)+m$ and $\lambda_{H,-}(m)$ denote the eigenvalue associated with $-\partial_t-\Delta+H(\n_x\cdot)+m$. We say that $H$ is reversible if, for any $m\in \mathscr C^{1,2}_{\mathrm{per}}(\TT) $, $\lambda_{H,+}(m)=\lambda_{H,-}(m)$.
\end{definition}
The reversibility of quadratic Hamiltonians is a consequence of the Hopf-Cole transform; indeed, supposing that  $H$ writes $H(p)=|p|^2$, this amounts to investigating whether the first eigenvalue $\lambda(m)$ of 
\[ \partial_tu_m-\Delta u_m=\lambda(m)u_m+mu_m\] is equal to the first eigenvalue $\lambda_-(m)$ of 
\[ -\partial_t v_m-\Delta v_m=\lambda_-(m)v_m+mv_m.\] However, this simply expresses the fact that a linear operator and its adjoint have the same spectral radius. We claim that, from the point of view of reversibility, quadratic Hamiltonians are exceptional in the following sense:
\begin{theorem}\label{Th:Reversibility} We work in one-dimension ($d=1$).
Let $H$ satisfy \eqref{Eq:HypH}, and be analytic in the sense that 
\[\forall p\in \R\,, H(p)=\sum_{k=0}^\infty a_kp^k\]and assume that for any $\e>0$  the Hamiltonian $\e H$ is reversible.  Then $H$ is quadratic: up to a multiplicative constant, $H(p)=p^2$. \end{theorem}
This property justifies certain of the upcoming restrictions.

\subsection{Monotonicity with respect to parameters}
\paragraph{Increasing the frequency in time}
Let for any $\tau>0$ $\lambda_H(\tau,m)$ be the effective Hamiltonian associated with the operator $\tau\partial_t-\Delta +H(\n_x\cdot)+m$.
\begin{theorem}\label{Th:Frequency} Assume that either $H$ is reversible or that $m$ is symmetric in time ($m(T-\cdot)=m$).
Then the map $\tau\mapsto \lambda_H(\tau,m)$ is non-increasing. It is monotone decreasing \textcolor{black}{unless $m$ writes 
\[ m(t,x)=m_0(x)+h(t)\] for some periodic functions $m_0\,, h$.}\end{theorem}
In the case of quadratic hamiltonians, this theorem was obtained by \textcolor{black}{Liu}, Lou, Peng \& Zhou \cite{zbMATH07183857,zbMATH07122717} using a functional approach, which we claim is a corollary of the Donsker-Varadhan formula \eqref{Eq:DonskerVaradhan}. Our proof provides what we deem a more direct insight into the computations of \cite{zbMATH07183857,zbMATH07122717} and, we feel, places a stronger emphasis on the particular structure of reversibility. 
\begin{remark}[A conjecture]
It is unclear at this stage whether this monotonicity property holds for non-reversible hamiltonians. We leave this as an open and interesting question.
\end{remark}

An elementary result which we will use is the following consequence of the Donsker-Varadhan formula, already used implicitly as the main ingredient in \cite{zbMATH07183857,zbMATH07122717}:
\begin{lemma}\label{Le:QuantifiedDV}
Let $(\lambda_H\,, \p_H)$ be the eigenpair associated with the Hamiltonian $H$ and $\eta_H$ be the associated invariant measure. For any $\p\in \mathscr C^{1,2}(\T)$, there exists a vector field $\xi$ such that 
\begin{multline*} \iint_\TT \left(\partial_t\p-\Delta\p+H(t,x,\n_x\p)\right)d\eta_H-\lambda_H\\=\frac12\iint_\TT \n^2_{pp}H(\xi)\left[\p-\p_H,\p-\p_H\right]d\eta_H.\end{multline*} In particular,
\begin{multline*}\iint_\TT \left(\partial_t\p-\Delta\p+H(t,x,\n_x\p)\right)d\eta_H\\>\iint_\TT \left(\partial_t\p_H-\Delta\p_H+H(t,x,\n_x\p_H)\right)d\eta_H\end{multline*} unless $\p-\p_H$ is constant.
\end{lemma}

\paragraph{Non-monotonicity with respect to the diffusion} Similar to the quadratic case, where it is due to Carr\`ere \& Nadin \cite{zbMATH07226744}, we have the following result:
\begin{theorem}\label{Th:CarrereNadin}Let, for any $\mu>0$ and any $m\in \mathscr C^{1,2}_{\mathrm{Per}}(\TT)$ $\lambda_H(m,\mu)$ be the effective Hamiltonian associated with $\partial_t+\mu(-\Delta+H(\n_x\cdot))+m$. There exists $m\in \mathscr C^{1,2}_{\mathrm{per}}(\TT)$ such that $\mu\mapsto \lambda_H(m,\mu)$ is \textcolor{black}{non-monotonic}.
\end{theorem}
As in \cite{zbMATH07226744} the core ingredient of the proof is Theorem \ref{Th:BeltramoHess}.

\subsection{An example of qualitative property for elliptic problems with advection}
In this section, we  consider an elliptic problems with advection, which can be seen as a generalisation of the parabolic case. Throughout, we let $A\in \mathscr C^2(\T)$ be a smooth vector field. 
\begin{definition}\label{De:InvariantMeasureA}
An $A$-invariant measure is a measure $\eta\in \mathcal P(\T)$ such that
\[ \forall \phi\in \mathscr C^2(\T)\,, \int_\T \langle A,\n\phi\rangle d\eta=0\] or, alternatively, such that $-\nabla\cdot(A\eta)=0$ in the sense of distribution. The set of $A$-invariant measures is denoted by $\mathcal P_A(\T)$.
\end{definition}
Our main assumption is the following:
\begin{equation}\label{Eq:AssumptionVectorField}
\begin{cases}
\text{There exists a smooth positive measure $\eta_A\in \mathcal P_A(\T)$}\,, 
\\\text{The set of $\mathscr C^2(\T)$ $A$-invariant measures is dense for the weak-$*$ topology in $\mathcal P_A(\T)$.}
\end{cases}
\end{equation}
The first part of \eqref{Eq:AssumptionVectorField} is satisfied for divergence-free vector fields. Some sufficient conditions to ensure the validity of the first point (typically, if $A$ is divergence-free) were studied, we refer to \cite{zbMATH07147349}.

\paragraph{The low-diffusivity asymptotics}
We let, for any $\e>0$, $(\lambda_H(\e,A),\p_\e)$ be the eigenpair associated with the operator $\e(-\Delta+H(\n_x))+\langle A,\n_x\cdot\rangle+m$.
\begin{theorem}\label{Th:LowDiffusivityAdvection} \textcolor{black}{Under assumption \eqref{Eq:AssumptionVectorField},
there holds
\begin{equation}\label{Eq:LowDiffusivityAdvection}
\lambda_H(\e,A)\underset{\e\to 0}\rightarrow -\max_{\eta\in \mathcal P_A(\T)}\int_\T md\eta.\end{equation}}
\end{theorem}

\paragraph{Some comments} When $A=0$ and $m$ is independent of time, we recover Theorem \ref{Th:LowDiffusion2}. 
A case of special interest is that of shear flows that is, $A(x,y)=(u(y),0)$, where the following result holds \cite[Theorem 1.1]{Liu}:
\[ \lim_{\e\to 0}\lambda_H(A,\e)=-\max\left(\max_{y \in \T\,, u(y)\neq 0} \fint_\T m(x,y)dx\,, \max_{(x,y)\,, x\in \T y\,, u(y)=0}m(x,y)\right).\] \textcolor{black}{It is  unclear and unlikely that  assumption \eqref{Eq:AssumptionVectorField} is  satisfied in this form. Whether our approach could work in this case is an interesting question that we leave open for future research}. Another case that is not covered by our result is that of gradient-like vector field $A$, typically $A=\n f$ for a smooth function $f$. In this case, it is not clear that the entire sequence $\lambda_H(\e,A)$ converges as $\e\to 0$; we refer for more details to \cite{bai2023counterexampleprincipaleigenvalueelliptic}. As noted in \cite{bai2023counterexampleprincipaleigenvalueelliptic}, knowing the specific structure of the vector field that could guarantee the convergence of the eigenvalues is a question of great interest, that we leave open as the topic of future research.

\subsection{Related literature}\label{Se:Biblio}
The existence of an eigenpair $(\lambda_H,\p_H)$ and its use in the study of long-time asymptotics of Hamilton-Jacobi equations has been studied extensively in the first and second-order case (that is, with or without diffusion), starting with the seminal paper \cite{LionsPapanicolaouVaradhan}. We refer for instance to \cite{Barles1985,zbMATH02130241,Chasseigne_2019,contreras2013weak,Evans:2001aa,zbMATH06797741,zbMATH05177907,Namah_1999,zbMATH01595998,zbMATH01851000} for a study of both cases (observe that in the first-order case the long-time convergence of the associated Hamilton-Jacobi equation is not true in the time-periodic case). A particularly important reference for us is the article of Barles \& Souganidis \cite{zbMATH01578865}. Regarding the qualitative properties of eigenvalues in the case of quadratic Hamiltonians, that is, when $\lambda_H$ can be interpreted as the first eigenvalue of a linear operator (see \eqref{Eq:Main}), the interest in such qualitative properties lies in the interpretation of $\lambda_H$ as a survival criterion for a population living in a domain $\T$ with a time-periodic resources distribution $m$. We refer to \cite[Chapter 2]{zbMATH07668634} for this interpretation. Starting with the works of Beltramo \& Hess (see \cite{zbMATH00049232} for a comprehensive survey, as well as \cite{zbMATH01028210,zbMATH01525790} for more recent references), a great deal of energy was devoted to understanding the influence of the frequency in time of the equation \cite{zbMATH01579419}, of the intensity of the advection field \cite{zbMATH02196652}, on the impact of diffusion \cite{zbMATH07226744}...Several questions remained open after these works, and a breakthrough into several of them was obtained by Liu, Lou, Peng \& Zhou who, first of all, established a saddle-point like formula to derive numerous monotonicity properties and, second, investigated a large number of limit behaviours for such eigenvalue problems. We refer to \cite{zbMATH07668634,zbMATH07183857,zbMATH07122717} and the references therein. Observe that the present article offers generalisation of their results, with significantly different approaches that we believe sheds new light on their central contributions.

\section{Basic properties: proofs of Theorems \ref{Th:Basic}-\ref{Th:BeltramoHess}}
\begin{proof}[Proof of Theorem \ref{Th:Basic}] For the sake of notational convenience, we write $(\lambda,\p)$ for $(\lambda_H\,, \p_H)$. To establish the left inequality, let $(t^*,x^*)$ be a point of minimum of  $\p$. In particular, using $\eta=\delta_{(t^*,x^*)}$ in \eqref{Eq:DonskerVaradhan}, we obtain $m(t^*,x^*)\geq -\lambda$, yielding the left-hand side inequality. To obtain the right-hand side inequality, let $(\xi,\overline \p)$ be the eigenpair attached with $-\Delta+H(\n_x\cdot)+\fint_0^Tm(t,\cdot)dt$ and let $\overline \eta$ be the associated invariant measure. By the Jensen inequality, since $H$ is convex and $\overline\eta$ does not depend on $t$,  $\int_\T\fint_0^TH(\n_x\p)d\overline\eta\geq \int_\T H\left(\fint_0^T \n\p_x\right)d\overline\eta$. Thus, applying \eqref{Eq:DonskerVaradhan} twice, if we let $\tilde\p:=\fint_0^T \p(t,\cdot)dt$, $\tilde m=\fint_0^T m(t,\cdot)dt$ and introduce the probability measure $\tilde\eta=\frac1T\overline\eta$ on $\TT$ we obtain 
\[ -\xi\leq \int_\T (-\Delta \tilde\p+H(\n_x\tilde\p)+\tilde m)d\overline\eta\leq \iint_\TT (\partial_t\p-\Delta\p+H(\n_x\p)+m)d\tilde\eta\leq-\lambda.\] This concludes the proof.

\end{proof}

\begin{proof}[Proof of Theorem \ref{Th:BeltramoHess}] We use the following fact from \cite[Proof of Proposition 15.4]{zbMATH00049232}: from assumption \eqref{Eq:AssumptionBeltramoHess} there exists a smooth periodic curve $\gamma:[0;T]\to \T$ such that $\fint_0^T m(t,\gamma(t))dt>0$. Let $\overline\eta_\gamma$ be the measure induced by the graph of $\gamma$: for any continuous function $f$, $\langle f,\overline\eta_\gamma\rangle:=\fint_0^T f(t,\gamma(t))dt$. For a fixed, smooth, positive symmetric convolution kernel $\rho\in \mathscr C^{\infty,\infty}_{\mathrm{per}}(\TT)$ we let, for any $\delta>0$, $\rho_\delta:=\frac1{\delta^{d+1}}\rho(\frac{t}\delta,\frac{x}{\delta})$ and we define $\overline\eta_{\gamma,\delta}:=\overline\eta_\gamma*\rho_\delta.$ \textcolor{black}{This is to be understood as the periodic convolution product between a measure and a smooth function, defined as 
\[\overline\eta_{\gamma,\delta}:(t,x)\mapsto \iint_{\TT}\rho_\delta(t-s,x-y)d\overline\eta_\gamma(s,y)\] where the functions are extended by periodicity.}  As $\overline\eta_{\gamma,\delta}\underset{\delta\to0}\rightharpoonup \overline\eta_\gamma$ weakly in the sense of measures, for $\delta$ small enough there holds 
\[ \iint_\TT md\overline\eta_{\gamma,\delta}>0.\] Furthermore, $\overline\eta_{\gamma,\delta}$ is a smooth, positive function. Fix such a $\delta$. Let, for any $t\in [0;T]$, $\beta$ be the unique solution of 
\[-\nabla\cdot(\overline\eta_{\gamma,\delta}\n_x \beta)=\partial_t\overline\eta_{\gamma,\delta}+\Delta \overline\eta_{\gamma,\delta}.\]Defining $\alpha_\delta:=\n_x\beta$ it follows from elliptic regularity that $\alpha_\delta\in L^\infty_{\mathrm{per}}(0,T;\R^d)$, and, by definition 
\[ -\partial_t\overline\eta_{\gamma,\delta}-\Delta\overline\eta_{\gamma,\delta}-\nabla_x\cdot(\overline\eta_{\gamma,\delta}\alpha_\delta)=0.\] From \eqref{Eq:OptimalControl} we deduce that 
\[ -\e\lambda_H(\e)\geq \iint_\TT\left(-\e L(\alpha_\delta)+m\right)d\overline\eta_{\gamma,\delta}.\] Passing to the limit $\e\to 0$ yields $\lim\inf_{\e\to 0}(-\e\lambda_H(\e))\geq  \iint_\TT md\overline\eta_{\gamma,\delta}>0,$ whence the conclusion.
\end{proof}

\section{Scaling limits: proofs of Theorems \ref{Th:AsymptoticsFrequency}-\ref{Th:LowDiffusion2}-\ref{Th:LargeHeat}}
\begin{proof}[Proof of Theorem \ref{Th:AsymptoticsFrequency}]
For the sake of notational simplicity we let, for any $\tau>0$, $(\lambda(\tau),\p_\tau)$ be the eigenpair attached with $\tau\partial_t-\Delta+H(\n_x\cdot)+m$ and $\eta_\tau$ be the associated invariant measure.

We begin with the large frequency limit $\tau\to \infty$.  We let $(\overline\p,\xi)$ be the eigenpair attached with $-\Delta+H(\n_x\cdot)+\fint_0^T m(t,\cdot)dt$ and $\overline\eta$ be the associated invariant measure (in particular, $\overline\eta$ is independent of time).
 From Theorem \ref{Th:Basic} there holds $-\xi\leq -\lambda(\tau)$ for any $\tau>0$.   
 We now prove that $-\xi\geq -\lim_{\tau\to \infty}\lambda(\tau)$. Let $\eta_\infty$ be a closure point (for the weak-$*$ topology on measures) of $\{\eta_\tau\}_{\tau\to \infty}$ and $\lambda_\infty$ be a closure point of $\{\lambda(\tau)\}_{\tau \to \infty}$. Dividing \eqref{Eq:DonskerVaradhan} by $\tau$ and passing to the limit $\tau \to \infty$ we obtain that for any $f\in \mathscr C^{1,2}_{\mathrm{per}}(\TT)$, 
\[ 0\leq \iint_\TT (\partial_t f)d\eta_\infty.\] In other words, the map $\mathscr C^{1,2}_{\mathrm{per}}(\TT)\ni f\mapsto \iint_\TT \partial_t f d\eta_\infty$ is bounded from below and is thus identically zero. Consequently, in the sense of distributions, 
\[ \partial_t \eta_\infty=0.\] Furthermore, using $\phi=\overline\p$ as a test function in \eqref{Eq:DonskerVaradhan}, we have 
\[ -\lambda(\tau)\leq-\xi+ \iint_\TT \left(m(t,x)-\fint_0^T m(s,x)ds\right)d\eta_\tau(t,x).\] Passing to the limit  provides  (as $\iint_\TT ( m(t,x)-\fint_0^T m(s,x)ds)d\eta_\infty=0$)
\[ \lim_{\tau\to\infty}(-\lambda(\tau))\leq -\xi\] and the proof of \eqref{Eq:LargeFrequency} is concluded.

We continue with the low frequency limit. For any $t\in (0;T)$, we let $(\xi(t),\p_t)$ be the eigenpair attached to $-\Delta +H(\n_x\cdot)+m(t,\cdot)$ and $\eta_t$ be the associated invariant measure. First of all, define the test function $\phi:(t,x)\mapsto \p_t(x)$. Using $\phi$ in \eqref{Eq:DonskerVaradhan} we obtain 
\[ -\lambda(\tau)\leq \tau\iint_\TT \partial_t\phi d\eta_\tau-\fint_0^T \xi(t)dt\Rightarrow -\lim_{\tau\to 0}\lambda(\tau)\leq-\fint_0^T\xi(t)dt.\]We prove the converse inequality.
 As a consequence of the regularity of $m$, $(t,x)\mapsto \eta_t(x)$ lies in $\mathscr C^{1,2}_{\mathrm{per}}(\TT)$. Using the measure $\overline\eta$ defined as $\iint_\TT fd\overline\eta:=\fint_0^T \left(\int_\T f(t,x)d\eta_t(x)\right)dt$ in \eqref{Eq:DonskerVaradhan} we obtain 
 \[ -\iint_\TT (\tau\p_\tau)\partial_t\overline\eta+\fint_0^T\left(\int_\T (-\Delta \p_\tau+H(\n_x\p_\tau)+m(t,\cdot))d\eta_t(x)\right)dt\leq -\lambda(\tau)\] so that
 \begin{equation}\label{Eq:Sarajevo} -\iint_\TT (\tau\p_\tau)\partial_t\overline\eta +\fint_0^T \xi(t)dt\leq -\lambda(\tau).\end{equation} As $\iint_\TT H(\n_x\p_\tau)$ is uniformly bounded \textcolor{black}{as $\tau\to 0$}, we deduce from \eqref{Eq:HypH} that $\iint_\TT H(\tau\n_x\p)\underset{\tau\to 0}\rightarrow 0$. If we set $\omega_\tau:=\tau\p_\tau$, using the function $\mathcal H$ of assumption \eqref{Eq:HypH} and its convexity this implies $\mathcal H(\fint_0^T \Vert  \omega_\tau(t,\cdot)-\fint_\T \omega_\tau(t,\cdot)\Vert_{L^1(\T)})\leq \fint_0^T \mathcal H(\Vert \omega_\tau(t,\cdot)-\fint_\T \omega_\tau(t,\cdot)\Vert_{L^1(\T)})\underset{\tau\to 0}\rightarrow 0$. From the coercivity of $\mathcal H$ we deduce that $\fint_0^T \Vert  \omega_\tau(t,\cdot)-\fint_\T \omega_\tau(t,\cdot)\Vert_{L^1(\T)}\underset{\tau\to 0}\rightarrow 0$.  As $\frac{d}{dt}\int_\T \overline\eta(t,\cdot)=0$, we have 
 \begin{multline*}\left|\iint_\TT \omega_\tau\partial_t\overline\eta\right|=\left|\int_0^T\int_\T\left( \omega_\tau-\fint_\T \omega_\tau(t,\cdot)\right)\partial_t\overline\eta\right|\\ \leq \Vert \partial_t\overline\eta\Vert_{L^\infty}\int_0^T \left\Vert \omega_\tau(t,\cdot)-\fint_\T \omega_\tau(t,\cdot)\right\Vert_{L^1(\T)}\underset{\tau\to 0}\rightarrow 0.\end{multline*} Passing to the limit in \eqref{Eq:Sarajevo} proves $-\fint_0^T\xi(t)dt\leq -\lim_{\tau\to 0}\lambda(\tau)$. The proof is concluded.
\end{proof}

\begin{proof}[Proof of Theorem \ref{Th:LowDiffusion2}]
For the sake of notational simplicity, we let, for any $\mu>0$, $(\lambda(\mu),\p_\mu)$ be the eigenpair associated with $\partial_t+\mu(-\Delta+H(\n_x\cdot))+m$ and $\eta_\mu$ be the associated invariant measure. We begin with the low diffusivity limit $\mu\to 0$.  From \eqref{Eq:HypH}, $\mu H(\n_x\p)\geq H(\mu^\beta \n_x\p)$ for some $\beta\in(0;1)$.
Let $\psi_\mu:=\mu^{\beta} (\p_\mu-\fint_\T \p_\mu)$. Observe that integrating the equation 
\[ \partial_t\p_\mu-\Delta\p_\mu+H(\n_x\psi_\mu)+m\leq -\lambda(\mu)\] we obtain
\[ \fiint_\TT H(\n_x\psi_\mu)\leq -\lambda(\mu)-\fiint_\TT m \] whence from the Poincar\'e inequality, $\int_0^T \Vert \psi_\mu(t,\cdot)\Vert_{L^1(\T)}dt$ is uniformly bounded \textcolor{black}{ as ${\mu\to 0}$}. For any $\eta\in \mathscr C^2(\T)$ (in particular, time-independent) introduce $\tilde\eta:=(1/T)\eta\in \mathcal P(\TT)$. \eqref{Eq:DonskerVaradhan} yields
\[ -\mu^{1-\beta}\iint_\TT \psi_\mu \Delta\tilde\eta+\iint_\TT H(\n_x\psi_\mu)d\tilde\eta+\iint_\TT md\tilde\eta\leq -\lambda(\mu).\] Passing to the limit $\mu\to 0$ gives: for any $\eta\in \mathcal P(\T)\cap \mathscr C^2(\T)$, 
$\int_\T \left(\fint_0^T m\right)d\eta\leq -\lim_{\mu\to 0}\lambda(\mu).$ Consequently, 
\begin{equation}\label{Eq:Sorcerer1}\forall \eta\in \mathcal P(\T)\,,\int_\T \left(\fint_0^T m\right)d\eta\leq -\lim_{\mu\to 0}\lambda(\mu).\end{equation} On the other hand, let $\tilde\eta_0$ be a closure point (for the weak-$*$ topology on $\mathcal P(\TT)$) of $\{\eta_\mu\}_{\mu\to 0}$. For any $f\in \mathscr C^{1,2}(\TT)$, passing to the limit $\mu\to 0$ in the right-hand side of \eqref{Eq:DonskerVaradhan} proves that 
\[ \forall f\in \mathscr C^{1,2}_{\mathrm{per}}(\TT)\,, -\lim_{\mu\to 0}\lambda(\mu)\leq \iint_\TT\left(\partial_t f+m\right)d\tilde\eta_0.\] Consequently, $\mathscr C^{1,2}_{\mathrm{per}}(\TT)\ni f\mapsto \iint_\TT \partial_t fd\tilde\eta_0$ is a bounded from below linear map. It is thus identically zero, so that $\tilde\eta_0$ does not depend on $t$ and, with a slight abuse of notations, $\tilde\eta_0(t,x)=(1/T)\eta_0(x)$; in particular: 
\begin{equation}\label{Eq:Sorcerer2}
-\lim_{\mu\to 0}\lambda(\mu)\leq \int_\T \left(\fint_0^T m(t,x)dt\right)d\eta_0(x).
\end{equation}  Combined with \eqref{Eq:Sorcerer1} we obtain first that \[\int_\T \left(\fint_0^T m(t,x)dt\right)d\eta_0(x)=\max_{\eta\in \mathcal P(\T)}\int_\T \left(\fint_0^T m(t,x)dt\right)d\eta(x)=\max_{x\in \T}\fint_0^T m(t,x)dt,\] and second that $\lim_{\mu\to 0}\lambda(\mu)=- \int_\T \left(\fint_0^T m(t,x)dt\right)d\eta_0(x)$. The proof is concluded.

We continue with the large diffusivity limit $\mu\to \infty$. Using the uniform probability measure in \eqref{Eq:DonskerVaradhan} we know that for any $\mu>0$ we have $\fiint_\TT m\leq -\lambda(\mu)$. Second, let $\tilde\eta_\infty$ be a closure point (for the weak-$*$ topology on $\mathcal P(\TT)$) of $\{\eta_\mu\}_{\mu\to\infty}$. Dividing \eqref{Eq:DonskerVaradhan} by $\mu$ and passing to the limit $\mu\to \infty$ we know that, for any $f\in \mathscr C^{1,2}_{\mathrm{per}}(\TT)$ there holds $\iint_\TT\left(-\Delta f+H(\n_x f)\right)d\tilde\eta_\infty\geq 0$. From \eqref{Eq:HypH}, choosing, for any $\e>0$, $f=\e g$ with $g\in \mathscr C^{1,2}_{\mathrm{per}}(\TT)$, this yields:
\[\forall g \in \mathscr C^{1,2}_{\mathrm{per}}(\TT)\,, -\iint_\TT \Delta g d\tilde\eta_\infty\geq 0.\] This implies that $\tilde\eta_\infty$ does not depend on $x$, and we write with an abuse of notation $\tilde\eta_\infty=(1/|\T|)\eta_\infty(t)$. Then, using \eqref{Eq:DonskerVaradhan} with a $x$-independent function $\phi$ we derive, passing to the limit $\mu\to \infty$, that for any $\phi\in \mathscr C^1_{\mathrm{per}}((0;T))$, 
\[-\lim_{\mu\to\infty}\lambda(\mu)\leq \int_0^T \phi'(t)d\eta_\infty(t)+\iint_\TT md\tilde\eta_\infty.\] As a consequence, $\mathscr C^1_{\mathrm{per}}((0;T))\ni\phi\mapsto \int_0^T \phi'd\eta_\infty$ is a linear, bounded from below map. It is thus identically zero, so that $\eta_\infty$ does not depend on $t$. Thus, $\eta_\infty$ is the uniform probability measure on $[0;T]$. Finally, taking $\phi\equiv 1$ in \eqref{Eq:DonskerVaradhan} and passing to the limit $\mu\to \infty$ yields $-\lim_{\mu\to\infty}\lambda(\mu)\leq \iint_\TT md\tilde\eta_\infty=\fiint_\TT m$, and the proof is finished.

 \end{proof}
 
 \begin{proof}[Proof of Theorem \ref{Th:LargeHeat}]
 For the sake of notational simplicity we let, for any $\e>0$, $(\lambda(\e),\p_\e)$ be the eigenpair attached to $\partial_t-\Delta+\e H(\n_x\cdot)+m$ and $\eta_\e$ be the associated invariant measure. 
 \textcolor{black}{Observe that using $\p=\psi_0$, the solution of \eqref{Eq:Lim}, in \eqref{Eq:DonskerVaradhan} yields}
 \begin{equation}\label{Eq:MC}-\lambda(\e)\leq \e\iint_\TT H(\n_x\psi_0)d\eta_\e+\fiint_\TT m.\end{equation} By parabolic regularity, $\psi_0\in \mathscr C^{1,2}_{\mathrm{per}}(\TT)$ so that  $\lim_{\e\to 0}(-\lambda(\e))=\fiint_\TT m$.
 
 We let $\eta_0$ be a closure point (for the weak-$*$ topology on $\mathcal P(\TT)$) of the sequence $\{\eta_\e\}_{\e\to 0}$. For any $h\in \mathscr C^{1,2}_{\mathrm{per}}(\TT)$ with $\fiint_\TT h=0$ let $\psi_h$ be the solution of $\partial_t \psi_h-\Delta \psi_h=\fiint_\TT m- (m+h)$ endowed with periodic boundary conditions and $\fiint_\TT \psi_h=0$. \eqref{Eq:DonskerVaradhan} with $\phi=\psi_h$ gives, for any such $h$, 
 \[\lim_{\e\to 0}(-\lambda(\e))\leq -\iint_\TT h d\eta_0+\fiint_\TT m\] so that the map $\mathscr C^{1,2}_{\mathrm{per}}\cap \{f\,, \fiint_\TT f=0\}\ni h\mapsto \iint_\TT hd\eta_0$ is bounded from below. In particular, it is identically zero and thus $\eta_0$ is the uniform measure on $\mathcal P(\TT)$. Now, turning back to \eqref{Eq:MC}, we see that $\fiint_\TT H(\n_x\p_\e)$ is uniformly bounded \textcolor{black}{as $\e\to 0$}. From \eqref{Eq:HypH} this implies that $\{\p_\e\}_{\e\to 0}$ is uniformly bounded in $L^r(0,T;W^{1,r}(\T))\subset (L^q(0,T;W^{1,q}(\T)))'$ for some $r,q>1$. Indeed, integrating the equation in $x$ yields that $t\mapsto \fint_\T \p_\e(t,\cdot)$ is (uniformly in $\e$) $BV$ in $t$, whence $\sup_{t\in[0;T]}|\fint_\T\p_\e|$ is uniformly bounded. This implies the weak-$*$ convergence (for the weak topology on $(L^q(0,T;W^{1,q}(\T)))'$) of $\{\p_\e\}_{\e\to 0}$ to a $\p_0$. Passing to the limit in the equation we deduce that for any $\eta\in \mathscr C^{1,2}_{\mathrm{per}}(\TT)\cap \mathcal P(\TT)$ there holds $\iint_\T \p_0(-\partial_t\eta-\Delta \eta)+\iint_\TT (m-\fiint_\TT m)\eta=0$ whence $\p_0=\psi_0$. Furthermore, $\{\n_x\p_\e\}_{\e\to 0}$ is uniformly bounded in 
 $L^r(\TT)$ and converges weakly (for the weak topology on $L^r(\TT)$) to some $\hat\p_0$. Using once again regular enough test functions in $L^r(\TT)$, we deduce that $\hat \p_0=\psi_0$. 
  As $H$ is convex, this implies
 %\footnote{Extending if necessary the definition of $H$ to $(L^q(0,T;W^{1,q}(\T)))'$ as follows: $\tilde H(f)=+\infty$ if $f\notin \mathrm{dom}(H)$, this extension is convex.}
 $\fiint_\TT H(\n_x\psi_0)\leq \underset{\e\to 0}{\lim\inf}\fiint_\TT H(\n_x\p_\e)$ so that
 \begin{multline*} \fiint_\TT H(\n_x\psi_0)\leq \underset{\e\to0}{\lim\inf}\fiint_\TT H(\n_x\p_\e)\leq\lim_{\e\to 0} \frac{-\lambda(\e)-\fiint_\TT m}{\e}\\\leq \lim_{\e\to 0}\iint_\TT H(\n_x\psi_0)d\eta_\e=\fiint_\TT H(\n_x\psi_0).\end{multline*}

  \end{proof}

 \section{Reversibility of Hamiltonians: proof of Theorem \ref{Th:Reversibility}}
 
 \begin{proof}[Proof of Theorem \ref{Th:Reversibility}]
 Assume that, for any $\e>0$, $\e H$ is reversible. From Theorem \ref{Th:LargeHeat} this implies the following: for any $m\in \mathscr C^{1,2}(\TTo)$ with $\fiint_\TTo m=0$, if $\psi_m\,, \theta_m$ denote the solutions of 
 \[\begin{cases}\partial_t\psi_m-\Delta \psi_m=m&\text{ in }\TTo\,, 
 \\ \psi_m(T,\cdot)=\psi_m(0,\cdot)&\text{ in }\mathbb T\,, 
 \\ \fiint_\TTo \psi_m=0\end{cases}\text{ and }\begin{cases}-\partial_t\theta_m-\Delta \theta_m=m&\text{ in }\TTo\,, 
 \\ \theta_m(T,\cdot)=\theta_m(0,\cdot)&\text{ in }\mathbb T\,, 
 \\ \fiint_\TTo \theta_m=0\end{cases}\] then 
 \[ \fiint_\TTo H(\n_x\psi_m)=\fiint_\TTo H(\n_x\theta_m).\] Indeed, an immediate adaptation of the proof of Theorem \ref{Th:LargeHeat} show that \[\frac{-\lambda_{H,-}(\e)-\fiint_\TT m}\e\underset{\e\to 0}\to\fiint_\TT H(\n_x\theta_m)
.\]Replacing $m$ with $\delta m$ for any $\delta>0$ and observing that $\psi_{\delta m}=\delta\psi_m\,, \theta_{\delta m}=\delta\theta_m$ this implies that 
\[ \sum_{k=0}^\infty a_k\delta^k \fiint_\TT (\partial_x\psi_m)^k=\sum_{k=0}^\infty a_k\delta^k \fiint_\TT (\partial_x\theta_m)^k\] whence:
\[\forall k\in \N\text{ such that } a_k\neq 0\,, \forall m\in \mathscr C^{1,2}_{\mathrm{per}}(\TT)\,, \fiint_\TT (\partial _x\psi_m)^k=\fiint_\TT (\partial_x \theta_m)^k.\] Naturally, the proof is concluded if we can prove the following: 
\begin{equation}\label{Eq:Akimbo} \forall k\in \N\setminus\{0,1,2\}\,, \exists m\in \mathscr C^{1,2}_{\mathrm{per}}(\TT)\, \fiint_\TT (\partial_x\psi_m)^k\neq \fiint_\TT (\partial_x\theta_m)^k.\end{equation}

To prove \eqref{Eq:Akimbo} let $k\in \N\,, k>2$. Let $(e_j,\xi_j)$ be the eigenpairs of the Laplacian on $\mathbb T$, that is, $-\Delta e_j=\xi_j e_j$ with $\{\xi_j\}_{j\in \N}$ non-decreasing and going to $+\infty$. Similarly, for any $c\in \mathscr C^0_{\mathrm{per}}((0;T))$ and any $j\in \N\setminus\{0\}$ let $\alpha_{j,c}$, resp. $\beta_{j,c}$, be the unique solution of 
\[\begin{cases}\alpha_{j,c}'+\lambda_j\alpha_{j,c}=c&\text{ in }(0;T)\,, 
\\ \alpha_{j,c}(T)=\alpha_{j,c}(0),\end{cases}\text{ resp. }\begin{cases}-\beta_{j,c}'+\lambda_j\beta_{j,c}=c&\text{ in }(0;T)\,, 
\\ \beta_{j,c}(T)=\beta_{j,c}(0).\end{cases}\] If $m$ writes $m(t,x)=\sum_{j=0}^\infty c_j(t)e_j(x)$, it follows that 
\[ \psi_m=\sum_{j=0}^\infty\alpha_{j,c_j}e_j\,, \theta_m=\sum_{j=0}^\infty \beta_{j,c_j}e_j.\]  We now fix $j\neq \ell$ such that $\fint_\T (\partial_x e_j)^{k-1}\partial_xe_\ell\neq 0$ and $\lambda_j\neq \lambda_\ell$ (this is always possible as $k>2$) and we choose $m=\bar c_j(t)e_j+\delta c_\ell(t)e_\ell$, where $\bar c_j\,,c_\ell\in \mathscr C^{1}_{\mathrm{per}}((0;T))$ are functions to be determined (but fixed) and $\delta>0$ is a positive parameter. If for any $m$ $\fiint_\TT (\partial_x\psi_m)^k=\fiint_\TT(\partial_x\theta_m)^k$ it follows, linearising around $\delta=0$, that $\fiint_\TT (\partial_xe_j)^{k-1}\partial_xe_\ell \alpha_{j,\bar c_j}^{k-1}\alpha_{\ell,c_\ell}=\fiint_\TT (\partial_xe_j)^{k-1}\partial_xe_\ell \beta_{j,\bar c_j}^{k-1}\beta_{\ell,c_\ell}$ whence we deduce that for any $\bar c_j\,, c_\ell\in \mathscr C^{1}_{\mathrm{per}}((0;T))$,
\[ \fint_0^T \alpha_{j,\bar c_j}^{k-1}\alpha_{\ell,c_\ell}=\fint_0^T \beta_{j,\bar c_j}^{k-1}\beta_{\ell,c_\ell}.\] Similarly, letting $\bar c_j(t)=1+\tau  c_j(t)$ for $\tau>0$, linearising the previous identity around $\tau=0$ yields that for all $c_j\,, c_\ell$, there holds
\[\fint_0^T \alpha_{j,c_j}\alpha_{\ell,c_\ell}=\fint_0^T \beta_{j,c_j}\beta_{\ell,c_\ell}.\] Let us prove that this is impossible when $\lambda_j\neq \lambda_\ell$. Introduce the functions $\eta_{\ell,j}\,, \xi_{\ell,j}$ as the solutions of 
\[
\begin{cases}
-\eta_{\ell,j}'+\lambda_j\eta_{\ell,j}=\alpha_{\ell,c_\ell}\,, 
\\ \xi_{\ell,j}'+\lambda_j\xi_{\ell,j}=\beta_{\ell,c_\ell}.\end{cases}\] It follows that for any $c_j$ we have 
$ \fint_0^T \eta_{\ell,j}c_j=\fint_0^T \xi_{\ell,j}c_j$, whence 
$\eta_{\ell,j}=\xi_{\ell,j}$ for any $c_\ell$. Thus, we deduce that 
\[ -\eta_{\ell,j}''+\lambda_j^2\eta_{\ell,j}=\left(\frac{d}{dt}+\lambda_j\right)\alpha_{\ell,c_\ell}=c_\ell+(\lambda_j-\lambda_\ell)\alpha_{\ell,c_\ell}\] and, similarly, 
\[-\eta_{\ell,j}''+\lambda_j^2\eta_{\ell,j}=-\xi_{\ell,j}''+\lambda_j^2\xi_{\ell,j}=c_\ell+(\lambda_j-\lambda_\ell)\beta_{\ell,c_\ell}.\] As a consequence, 
\[ (\lambda_\ell-\lambda_j)\alpha_{\ell,c_\ell}=(\lambda_\ell-\lambda_j)\beta_{\ell,c_\ell}.\] As $\lambda_j\neq \lambda_\ell$ by construction we deduce that we must have, for any $c_\ell$, $\alpha_{\ell,c_\ell}=\beta_{\ell,c_\ell}$. However, this implies $\lambda_{\ell}\alpha_{\ell,c_\ell}=c_\ell$ for any $c_\ell$, which is impossible. The proof is concluded.

 \end{proof}
 
\section{(Non-)Monotonicity properties: proofs of Theorems \ref{Th:Frequency}-\ref{Th:CarrereNadin}}
\begin{proof}[Proof of Theorem \ref{Th:CarrereNadin}]
Let $m_0\in \mathscr C^{1,2}_{\mathrm{per}}(\TT)$ be such that:
\[ \forall x \in \T\,, \fint_0^T m_0(t,x)dt=0\,, \fint_0^T \max_{x\in \T}m_0(t,x)dt>0.\] From Theorem \ref{Th:BeltramoHess} we know that for any $\mu>0$, if we let $\lambda_H(\frac1\e m_0,\mu)$ be the effective Hamiltonian associated with $\partial_t+\mu(-\Delta+H(\n_x ))+\frac1\e m_0$  then 
\[ \lim_{\e\to0}\lambda_H(\frac1\e m_0,1)=-\infty.\] We fix $\e_0>0$ such that $\lambda_H(\frac1{\e_0}m_0,1)\leq -1$ and we define $m:=\frac1{\e_0}m_0$. From Theorem \ref{Th:LowDiffusion2} there holds 
\[ \lim_{\mu\to 0}\lambda_H(m,\mu)=0\] and thus $\lambda_H(m,0)>\lambda_H(m,1)$. On the other hand, 
$\lim_{\mu\to \infty}\lambda_H(m,\mu)=0$ as well, whence the conclusion.
\end{proof}

Before we prove Theorem \ref{Th:Frequency} we show how to derive Lemma \ref{Le:QuantifiedDV}.
\begin{proof}[Proof of Lemma \ref{Le:QuantifiedDV}]
We let $(\lambda_H,\p_H)$ be the eigenpair and $\eta_H$ the associated invariant measure.  For any $\p\in \mathscr C^{1,2}(\TT)$ define $z:=\p-\p_H$. Then, letting $F(\p)=\partial_t\p-\Delta \p+H(\n_x \p)+m$ we obtain, by the mean-value theorem, the existence of a vector field $\xi$ such that 
\begin{multline*} \iint_\TT (F(\p)-F(\p_H))d\eta_H=\iint_\TT (\partial_t z-\Delta z+\langle \n_pH(\n_x\p),\n z\rangle)d\eta_H\\+\frac12\iint_\TT \n^2_{pp}H(\xi)[\n z,\n z]d\eta_H.\end{multline*} The conclusion follows since $\eta_H$ solves \eqref{Eq:DefInvariantMeasure} and as $F(\p_H)=\lambda_H$. By the same reasoning, if $H$ satisfies \eqref{Eq:HypH} and is thus in particular strictly convex we have 
\[\iint_\TT F(\p)d\eta_H>\lambda_H+\iint_{\{\n\p\neq \n\p_H\}} d\eta_H.\] By the maximum principle $\min_\TT \eta_H>0$ and so $\iint_\TT F(\p)d\eta_H=\lambda_H$ implies $|\{\n\p\neq\n \p_H\}|=0$. The proof is concluded.
\end{proof}

\begin{proof}[Proof of Theorem \ref{Th:Frequency}]By assumption, since either $H$ is reversible or $m$ is symmetric in time we have the following property: letting for any $\tau$ $(\lambda_{\pm}(\tau),\p_{\pm,\tau})$ be the eigenpair associated with the Hamiltonian $\pm\tau\partial_t-\Delta+H(\n_x\cdot)+m$ and $\eta_{\pm,\tau}$ be the associated invariant measure there holds $\lambda_+(\tau)=\lambda_-(\tau)$. Now, consider the Gateaux derivative $(\dot\lambda_+(\tau),\dot\p_{+,\tau})$ of the map $\tau\mapsto (\lambda_+(\tau),\p_{+,\tau})$. By direct computations, $(\dot\lambda_+(\tau),\dot\p_{+,\tau})$ solves $L \dot\p_{+,\tau}=-\dot\lambda_+(\tau)-\partial_t\p_{+,\tau}$ with $L= \tau\partial_t-\Delta +\langle \n_pH(\n_x\p_{+,\tau}),\n_x\rangle$. Letting $L^*=-\tau\partial_t-\Delta-\n\cdot(\n_pH(\n_x\p_{+,\tau})\cdot)$ be the adjoint of $L$, observe that $L^*\eta_{+,\tau}=0$. Finally, introducing $F(\p)=\tau\partial_t\p-\Delta \p+H(\n_x \p)+m$ and $\overline F(\p)=-\tau\partial_t\p-\Delta \p+H(\n_x \p)+m$ and noting that $2\tau\partial_t \p_{+,\tau}=F(\p_{+,\tau})-\overline F(\p_{+,\tau})$ we observe that 
\begin{align*}
\dot\lambda_+(\tau)&=\iint_\TT \partial_t\p_{+,\tau}d\eta_{+,\tau}
\\&=\frac1{2\tau}\iint_\TT (F(\p_{+,\tau})-\overline F(\p_{+,\tau}))d\eta_{+,\tau}
\\&=\frac1{2\tau}\iint_\TT (\overline F(\p_{-,\tau})-\overline F(\p_{+,\tau}))d\eta_{+,\tau}
\\&\text{ since by reversibility $F(\p_{+,\tau})=\lambda_+(\tau)=\lambda_-(\tau)=\overline F(\p_{-,\tau})$}
\\&=\frac1{2\tau}\iint_\TT (\overline F(\p_{-,\tau})-\overline F(\p_{+,\tau}))d\eta_{-,\tau}
\\&\text{ since by reversibility $\eta_{+,\tau}=\eta_{-,\tau}$}
\\&\leq 0
\end{align*}
where the last inequality is a consequence of Lemma \ref{Le:QuantifiedDV}.\textcolor{black}{From Lemma \ref{Le:QuantifiedDV}, this inequality is strict, unless $\p_{-,\tau}=\p_{+,\tau}+\bar\p_\tau(t)$ for some $T$-periodic function $T$. Should this be the case, this provides
\[ \p_\tau'(t)=-2\partial_t\p_{+,\tau},\] so that $\partial_t\p_{+,\tau}$ only depends on $t$. In particular, we can write $\p_{+,\tau}=f_0(t)+f_1(x)$ for two functions $f_0\,, f_1$, and this gives
\[\lambda_H(m)+f_0'(t)-\Delta f_1+H(\n_xf_1)=-m\] so that $m$ finally writes $m(t,x)=m_0(x)+h(t)$, as announced. 
}
\end{proof}

\section{Operators with advection: proof of Theorem \ref{Th:LowDiffusivityAdvection}}%-\ref{Th:Advection}}
\begin{proof}[Proof of Theorem \ref{Th:LowDiffusivityAdvection}]
To simplify notations we let, for any $\e>0$, $(\lambda(\e),\p_\e)$ be the unique eigenpair associated with $\e(-\Delta+H(\n_x))+\langle A,\n_x\rangle+m$. We let $\eta_\e\in \mathcal P(\T)$ be the associated invariant measure, that is, the solution of $\e(-\Delta-\n_x(\n_pH(\n_x\p_\e)\cdot))-\n_x\cdot(A\cdot)=0$. First of all, let $\eta_0\in \mathcal P(\T)$ be a closure point of $\{\eta_\e\}_{\e\to 0}$, and $\lambda_0$ be a closure point of $\{\lambda(\e)\}_{\e\to 0}$ (the existence of $\lambda_0$ is guaranteed by bounds similar to that of Theorem \ref{Th:Basic}). Let us prove that 
\[ \eta_0\in \mathcal P_A(\T).\] From \eqref{Eq:DonskerVaradhan}, for any $f\in \mathscr C^2(\T)$, $-\lambda(\e)\leq \int_\T \left(\e(-\Delta f+H(\n_x f))+\langle A,\n_xf\rangle+m\right)d\eta_\e.$ In particular, passing to the limit $\e\to 0$ we obtain 
\[ \forall f\in \mathscr C^2(\T)\,, -\lambda_0\leq \int_\T \langle A,\n_x f\rangle d\eta_0+\int_\T md\eta_0,\] whence the linear map $\mathscr C^2(\T)\ni f\mapsto \int_\T  \langle A,\n_x f\rangle d\eta_0$ is bounded from below. Thus, this map is identically zero, which means that $\eta_0\in \mathcal P_A(\T)$. Furthermore, this implies that 
\begin{equation}\label{Eq:LDA1}
-\lambda_0\leq \int_\T md\eta_0.\end{equation} Second, let us study the sequence $\{\p_\e\}_{\e\to 0}$. From \eqref{Eq:AssumptionVectorField} let $\eta_A$ be a positive, smooth probability measure associated with $A$. From \eqref{Eq:HypH}, $\e H(\n_x\p)\geq H(\e^\beta \n_x\p)$ for some $\beta<1$.
We obtain, letting $\psi_\e:=\e^{\beta} (\p_\e-\int_\T \p_\e)$, 
\[ -\e^{1-\beta}\int_\T \psi_\e \Delta \eta_A+\int_\T H(\n_x\psi_\e)d\eta_A+\int_\T md\eta_A\leq -\lambda(\e).\] %From the Poincar\'e inequality
%\[ \int_\T \left\vert f-\fint_\T f\right\vert\lesssim \int_\T H(\n_x f)\] 
From \eqref{Eq:HypH} and the fact that $\eta_A$ is positive we deduce that $\{ \int_\T H(\n_x\psi_\e)\}_{\e\to 0}$ is bounded, so that there exists $r>1$ and a strong $L^r(\T)$-closure point  $\psi_0$ of the sequence $\{\psi_\e\}_{\e\to 0}$. Now, let $\eta\in \mathcal P_A(\T)$ be any smooth $A$-invariant measure. The same computations lead to 
\[-\e^{1-\beta}\int_\T \psi_\e \Delta \eta+\int_\T md\eta\leq -\e^{1-\beta}\int_\T \psi_\e \Delta \eta +\int_\T H(\n_x\psi_\e)d\eta +\int_\T md\eta \leq -\lambda(\e).\] Passing to the limit $\e\to 0$ we obtain that 
for any  $\eta\in \mathcal P_A(\T)\cap \mathscr C^2(\T)\,, \int_\T md\eta\leq -\lambda_0.$ From assumption \eqref{Eq:AssumptionVectorField} and \eqref{Eq:LDA1} this provides
\[ \forall \eta\in \mathcal P_A(\T)\,, \int_\T md\eta\leq -\lambda_0\leq \int_\T md\eta_0.\] We thus deduce that $\eta_0$ maximises $\mathcal P_A(\T)\ni\eta\mapsto\int_\T md\eta$, and the conclusion follows.
\end{proof}
%
%\begin{proof}[Proof of Theorem \ref{Th:Advection}]
%
%\end{proof}
\bibliographystyle{abbrv}
\bibliography{BiblioHJB}

\end{document}